\documentclass[11pt]{amsart}
\usepackage{amssymb}
\usepackage{verbatim}

\numberwithin{equation}{section}

\newtheorem{lemma}[equation]{Lemma}
\newtheorem{question}[equation]{Question}
\newtheorem{thm}[equation]{Theorem}

\newtheorem{cor}[equation]{Corollary}
\newtheorem{prop}[equation]{Proposition}

\theoremstyle{remark}

\theoremstyle{definition}

\newtheorem{example}[equation]{Example}

\DeclareMathOperator{\charp}{{char}}

\newcommand{\Z}{{\mathbb Z}}

\newcommand{\Q}{{\mathbb Q}}

\newcommand{\R}{{\mathbb R}}

\newcommand{\A}{{\mathbb A}}

\newcommand{\PP}{{\mathbb P}}

\DeclareFontEncoding{OT2}{}{} 
  \newcommand{\textcyr}[1]{%
    {\fontencoding{OT2}\fontfamily{wncyr}\fontseries{m}\fontshape{n}%
     \selectfont #1}}
\newcommand{\Sha}{{\mbox{\textcyr{Sh}}}}

\usepackage[colorlinks,pagebackref,pdftex, bookmarks=false]{hyperref}

\begin{document}



\title[Some Diophantine equations]{Some Diophantine equations related to positive-rank elliptic curves}

\author{Gwyneth Moreland}
\address{
  Community High School,
  Ann Arbor, MI 48104,
  USA
}
\email{gwynsm@gmail.com}

\author{Michael E. Zieve}
\address{
  Department of Mathematics,
  University of Michigan,
  Ann Arbor, MI 48109--1043,
  USA
}
\email{zieve@umich.edu}
\urladdr{www.math.lsa.umich.edu/$\sim$zieve/}

\thanks{The first author thank Community High School for enabling her to work with the second author
via the Community Resource program.
The second author was partially supported by the NSF under grant
DMS-1162181.}

\begin{abstract}
We give conditions on the rational numbers $a,b,c$ which imply that
there are infinitely many triples $(x,y,z)$ of rational numbers such that
$x+y+z=a+b+c$ and $xyz=abc$.  We do the same for the equations
$x+y+z=a+b+c$ and $x^3+y^3+z^3=a^3+b^3+c^3$.  These results rely on exhibiting
families of positive-rank elliptic curves.
\end{abstract}

\date{4 April 2013}

\maketitle


\section{Introduction}

Several authors have studied the following question:

\begin{question} \label{q}
For which triples $(a,b,c)$ of pairwise distinct rational numbers does the system of equations
$x+y+z=a+b+c$, $xyz=abc$ have infinitely many solutions in rational numbers $x,y,z$?
\end{question}

In 1989, Kelly \cite{K2} showed that this system has infinitely many rational
solutions if $a,b,c$ are positive and satisfy certain easy-to-check conditions.
In 1996, Schinzel \cite{S} adapted an argument of Mordell's \cite{Mordell}
to give a different proof of Kelly's result in case $(a,b,c)=(1,2,3)$.
Recently Zhang and Cai \cite{ZC1} extended Schinzel's proof to the case
$(a,b,c)=(1,2,n)$ for any integer $n\ge 3$.  Our first goal is to answer
Question~\ref{q} in the greatest possible generality.  We obtain the following
result:

\begin{thm} \label{sumprod}
Let $a,b,c$ be pairwise distinct rational numbers such that,
for every permutation $(A,B,C)$ of $(a,b,c)$, we have
\begin{equation} \label{first}
A(B-C)^3 \ne B(C-A)^3
\end{equation}
and
\begin{equation} \label{second}
AB^2+BC^2+CA^2 \ne 3ABC.
\end{equation}
Then there are infinitely many triples $(x,y,z)$ of rational numbers 
such that $x+y+z=a+b+c$ and $xyz=abc$.
\end{thm}

There are infinitely many triples $(a,b,c)$ of pairwise distinct nonzero rational numbers
such that $a(b-c)^3=b(c-a)^3$; in fact we will exhibit all such triples
in Proposition~\ref{anotherparam}.
It seems unlikely that there is a simple numerical property of such a triple $(a,b,c)$ which
determines whether the system $x+y+z=a+b+c$, \,$xyz=abc$ has infinitely many solutions $(x,y,z)\in\Q^3$,
since we will show that this question is the same as determining whether an associated elliptic
curve $E_{abc}$ over $\Q$ has positive rank.  
However, we suspect that this system of equations
has infinitely many rational solutions for roughly half of
all triples $(a,b,c)$ of pairwise distinct nonzero rational numbers such that $a(b-c)^3=b(c-a)^3$.
We will provide numerical and heuristic
evidence for this belief in Section~\ref{fun}.  
Similar remarks apply for triples $(a,b,c)$ such that $ab^2+bc^2+ca^2=3abc$.

Our next result exhibits a situation
in which (\ref{first}) and (\ref{second}) automatically hold:

\begin{cor} \label{coprime}
Let $a,b,c$ be pairwise distinct integers which are pairwise coprime.
Then there are infinitely many triples $(x,y,z)$ of rational numbers
such that $x+y+z=a+b+c$ and $xyz=abc$.
\end{cor}

Kelly~\cite{K2} gave conditions on $a,b,c$ which ensure that the system of
equations $x+y+z=a+b+c$, \,$xyz=abc$ has
infinitely many \textit{positive} rational solutions.  We recover his result as a consequence of
Theorem~\ref{sumprod}:

\begin{cor}[Kelly] \label{sumprodpos}
Let $a,b,c$ be pairwise distinct positive rational numbers such that \emph{(\ref{first})} holds
for every permutation $(A,B,C)$ of $(a,b,c)$.  Then there are infinitely
many triples $(x,y,z)$ of positive rational numbers such that
$x+y+z=a+b+c$ and $xyz=abc$.
\end{cor}

The analogue of Corollary~\ref{coprime} for positive solutions is as follows:

\begin{cor} \label{coprimepos}
Let $a,b,c$ be pairwise distinct positive integers which are pairwise coprime.  Then there are
infinitely many triples $(x,y,z)$ of positive rational numbers such that
$x+y+z=a+b+c$ and $xyz=abc$.
\end{cor}

We will also prove analogues of Theorem~\ref{sumprod} and Corollary~\ref{sumprodpos} for the pair of
equations $x+y+z=a+b+c$ and $x^3+y^3+z^3=a^3+b^3+c^3$.  This system has been studied in the physics
literature, in the contex of zeros of $6j$ Racah coefficients \cite{BL}.  We will prove the following results.

\begin{prop} \label{cube}
Let $a,b,c$ be pairwise distinct rational numbers such that, for every permutation $(A,B,C)$ of $(a,b,c)$, we have
\begin{equation} \label{third}
(A+B)(A-B)^3 \ne (B+C)(B-C)^3
\end{equation}
and
\begin{equation} \label{fourth}
AB^2+BC^2+CA^2\ne A^3+B^3+C^3.
\end{equation}
Then there are infinitely many triples $(x,y,z)$ of rational numbers such that $x+y+z=a+b+c$ and $x^3+y^3+z^3=a^3+b^3+c^3$.
\end{prop}

\begin{prop} \label{poscube}
Let $a,b,c$ be pairwise distinct positive rational numbers such that every permutation $(A,B,C)$ of $(a,b,c)$ satisfies
\emph{(\ref{third})}.  Then there are infinitely many triples $(x,y,z)$ of positive rational numbers such that $x+y+z=a+b+c$
and $x^3+y^3+z^3=a^3+b^3+c^3$.
\end{prop}

Several authors have proved special cases of our results.  Besides the papers of Kelly \cite{K2},
Schinzel \cite{S}, and Zhang--Cai \cite{ZC1} mentioned previously, we note that Ren and Yang \cite{Ren}
proved Proposition~\ref{poscube} in the special case that $a$, $b$ and $c$ are three consecutive positive integers.
  Our results contradict several results in
the recent paper \cite{SS} by Sadek and El-Sissi.  The discrepancy stems from a mistake in
in the proof of \cite[Prop.~2.6]{SS}, where it is asserted
that the twelve points $P_{ij}$, $2P_{ij}$ (with $i\ne j$) are all distinct
from one another.  That is not always true, for instance it is not true when
$(a,b,c)=(3,10,24)$.  As a consequence, \cite[Prop.~2.6]{SS} and \cite[Thm.~2.7]{SS} are false, and the proof of
\cite[Thm.~3.1]{SS} is not valid.  We note, however, that the paper \cite{SS} contains interesting material
despite this mistake, for instance it uses this circle of ideas to produce high-rank elliptic curves.
For other recent work on related questions, see \cite{SSZ,U,ZC2,Z}.

This paper is organized as follows.  After some preliminary work in the next section,
we prove Theorem~\ref{sumprod} in Section~\ref{mainproof}.  Our proof crucially relies on Mazur's theorem on
rational torsion subgroups of elliptic curves \cite{M}.  In Section~\ref{fun} we prove Corollary~\ref{coprime} and discuss
Question~\ref{q} in the cases where Theorem~\ref{sumprod} does not apply.  We prove Corollaries~\ref{sumprodpos}
and \ref{coprimepos} in Section~\ref{secpos}, and in the final Section~\ref{seccube} we prove Propositions~\ref{cube}
and \ref{poscube}.


\section{From equal sums and products to ranks of elliptic curves}

In this section we translate Question~\ref{q} to 
the question of determining which elliptic curves in a certain infinite family
have positive rank.
For any $a,b,c\in\Q$, we write $s:=a+b+c$ and $p:=abc$.
Let $E_{abc}$ be the curve in $\PP^2$ whose affine equation is
\begin{equation} \label{Eabc}
v^2=u^3-\Bigl(\frac{s^4}{48}-\frac{sp}2\Bigr)u + \Bigl(\frac{s^6}{864}-\frac{s^3p}{24}+\frac{p^2}4\Bigr),
\end{equation}
and let $S_{abc}$ be the variety in $\A^3$ defined by $x+y+z=s$ and $xyz=p$.
If $p=0$ then the set of rational points $S_{abc}(\Q)$ is infinite, consisting of all permutations of all triples $(x,s-x,0)$ with $x\in\Q$.
In the more difficult case that  $p\ne 0$, we now give a precise connection between $S_{abc}(\Q)$ and $E_{abc}(\Q)$.

\begin{lemma}\label{E}
For $a,b,c\in\Q^*$, the
set of rational points $E_{abc}(\Q)$ contains
\[
I_{abc}:=\left\{\Bigl(\frac{s^2}{12}, \frac p2\Bigr), \,\Bigl(\frac{s^2}{12}, -\frac p2\Bigr), \,\mathcal{O}\right\},
\]
where $\mathcal{O}$ is the point $(0:1:0)$ in\/ $\PP^2$.
The function
\[
\rho\colon (x,y,z) \mapsto \Bigl( -\frac py + \frac{s^2}{12}, \,-\frac py(x+\frac y2 -\frac s2)\Bigr )
\]
defines a homeomorphism $\rho\colon S_{abc}(\R)\to E_{abc}(\R)\setminus I_{abc}$ whose restriction to $S_{abc}(\Q)$ induces
a bijection of $S_{abc}(\Q)$ with $E_{abc}(\Q)\setminus I_{abc}$.
\end{lemma}

\begin{proof}
This can be verified via a straightforward computation, in which one also verifies that
$\rho^{-1}((u,v))$ equals
\[
 \Biggl(
 \frac{ v+\frac 12 {su}-\frac 1{24} {s^3}+\frac 12 {p} }{u-\frac 1{12} {s^2}}, \,\frac{-p}{u-\frac 1{12} {s^2}}, 
 \, \frac{ -v+\frac 12 {su}-\frac 1{24} {s^3}+\frac 12 {p} }{u-\frac 1{12} {s^2}} \Biggr).\qedhere
\]
\end{proof}

In order to analyze whether the curve $E_{abc}$ has infinitely many rational points, we
now compute its genus.

\begin{lemma} \label{genus}
For $a,b,c\in\Q^*$, the
curve $E_{abc}$ has genus $0$ if $(a+b+c)^3=27abc$, and has genus $1$ otherwise.
\end{lemma}

\begin{proof}
Since the affine equation for $E_{abc}$ is a Weierstrass equation, it defines an irreducible
curve of genus $0$ or $1$.  Genus $0$ occurs if and only if $\Delta=0$, where
$\Delta:=p^3(s^2-27p)$ is the discriminant of the Weierstrass equation.
\end{proof}

We conclude this section by addressing the genus zero cases.  Our next result exhibits
the triples $(a,b,c)$ for which $E_{abc}$ has genus $0$.

\begin{lemma} \label{sing}
If $a,b,c\in\Q^*$ satisfy $(a+b+c)^3=27abc$, and $a,b,c$ are not all equal,
then there is a unique $t\in\Q\setminus\{0,1\}$ such that
\[
 a=c(t-1)^3\quad and \quad  b=-ct^3.
\]
Conversely, for any $c,t\in\Q^*$ with $t\ne 1$, the above equations define elements
$a,b\in\Q^*$ such that $(a+b+c)^3=27abc$ and $a,b,c$ are not all equal; moreover,
$a,b,c$ are pairwise distinct if and only if $t\notin\{-1,\frac 12,2\}$.
\end{lemma}

\begin{proof}
It is straightforward to verify the final sentence in the result.
Now fix $a,b,c\in\Q^*$ such that $(a+b+c)^3=27abc$, where $a,b,c$ are not all equal.
If $t\in\Q^*$ satisfies $a=r(t-1)^3$ \,and \,$b=-rt^3$, then $t^3=-b/c$, so there is at
most one choice for $t$.  It remains only to show that
there exists $t\in\Q\setminus\{0,1\}$ such that 
$a=r(t-1)^3$ \,and \,$b=-rt^3$.
We will show that these equations are satisfied for $t=(-a+2b-c)/(a+b-2c)$.
First, note that $a+b\ne 2c$: for, if $a+b=2c$ then
$(3c)^3=(a+b+c)^3=27abc$ implies $c^2=ab$, so that
$(a-b)^2=(a+b)^2-4ab=(2c)^2-4c=0$, which gives the contradiction
$a=b=(a+b)/2=c$.  Now it is straightforward to check that
\begin{align*}
a-c(t-1)^3 &= \frac{(a-c)((a+b+c)^3-27abc)}{c(a+b-2c)^3} \\
b+ct^3 &= \frac{(b-c)((a+b+c)^3-27abc)}{c(a+b-2c)^3},
\end{align*}
so that indeed $a=c(t-1)^3$ \,and \,$b=-ct^3$.
Next we show that our specified value of $t$ is neither $0$ nor $1$.
For, if $t=0$ then $a+c=2b$, and if $t=1$ then $b+c=2a$; either of these implies $a=b=c$
via the same argument we used to show that $a+b\ne 2c$.
%
This completes the proof.
\end{proof}

Finally, we determine $S_{abc}(\Q)$ when $E_{abc}$ has genus zero.

\begin{lemma} \label{singpar}
For any $c\in\Q^*$ and $t\in\Q\setminus\{-1,0,\frac 12,1,2\}$, put $a=c(t-1)^3$ and $b=-ct^3$.
Then $S_{abc}(\Q)\setminus\{c(t-t^2),c(t-t^2),c(t-t^2))\}$ equals
\[
\left\{\Bigl(\frac{ct(t-1)^3}{(u+1)(u+t)}, \,-\frac{ct(u+t)^2}{u+1}, \,\frac{ct(u+1)^2}{u+t}\Bigr)\colon
u\in\Q\setminus\{-1,-t\}\right\}.
\]
\end{lemma}

\begin{proof}
Fix $c\in\Q^*$ and $t\in\Q\setminus\{-1,0,\frac 12,1,2\}$, and put $a=c(t-1)^3$ \,and $b=-ct^3$.
For $A=a/c$ and $B=b/c$, the set $S_{abc}(\Q)$ is obtained from $S_{AB1}(\Q)$ by multiplying all coordinates
of all points by $c$.  Hence it suffices to prove the result in case $c=1$, and to simplify the notation we will assume
$c=1$ in what follows.  For any $u\in\Q\setminus\{-1,-t\}$, one easily checks that
\[
P_u:=\Bigl(\frac{t(t-1)^3}{(u+1)(u+t)}, \,-\frac{t(u+t)^2}{u+1}, \frac{t(u+1)^2}{u+t}\Bigr)
\]
is in 
$S_{abc}(\Q)$.
We have $P_u\ne (t-t^2,t-t^2,t-t^2)$, since otherwise by equating $x$-coordinates
we would obtain $(t-1)^2+(u+1)(u+t)=0$, which is a quadratic polynomial in $u$ whose
discriminant is the nonsquare $-3(t-1)^2$.
Now let $(x,y,z)$ be any point in $S_{abc}(\Q)$ which does not equal either $(a,b,c)$ or
$(t-t^2,t-t^2,t-t^2)$.  Since $P_0=(a,b,c)$, it suffices to prove that $(x,y,z)=P_u$ for some
$u\in\Q\setminus\{-1,-t\}$.  We will prove this for the value
\[
u := -\frac{2t^4 + t^3z - 2t^3 + tyz + ty - yz}{t(t^3 + tz - t + y)}.
\]
We first check that this expression for $u$ defines a rational number, by showing that its denominator is nonzero.
If $y=-t^3-tz+t$ then $x=a+b+c-y-z=(t-1)(z+t^2-2t)$,
so $-t^3(t-1)^3=abc=xyz=(t-1)(z+t^2-2t)(-t^3-tz+t)z$,
or equivalently
$t(t-1)(z-1)(z+t^2-t)^2=0$;
thus either $z=1$ or $z=t-t^2$, which imply that
$(x,y,z)$ is either $(a,b,c)$ or $(t-t^2,t-t^2,t-t^2)$, contradicting our hypothesis.
Next we check that $u\ne -1$: for, otherwise we would obtain
$y=-t^2+(t^2-t^3)z^{-1}$,
so $x=a+b+c-y-z=-z+(3t-2t^2)+(t^3-t^2)z^{-1}$
and the equation $xyz=abc$ implies that $z\in\{1,t-t^2\}$,
which again gives the contradiction
$(x,y,z)\in\{((t-1)^3,-t^3,1), \,(t-t^2,t-t^2,t-t^2)\}$.
The same reasoning shows that $u\ne -t$: for, if $u=-t$ then
$z=t+(t^4-t^3)y^{-1}$,
so from $x+y+z=a+b+c$ and $xyz=abc$ we obtain $y\in\{-t^3,t-t^2\}$,
giving the same contradiction as above.  Writing $P_u=(\hat x,\hat y,\hat z)$, one can check that
\begin{align*}
\hat x-x &= -\frac{(xyz-abc)(yz-ty+t^2z-t^4+t^3-t^2)}{(yz+t^2z+t^3-t^2)(yz-ty-t^4+t^3)},\\
\hat y-y &= \frac{(xyz-abc)(y+t^3)}{(y+tz+t^3-t)(yz+t^2z+t^3-t^2)}, \quad\text{ and}\\
\hat z-z &= \frac{t(xyz-abc)(z-1)}{(y+tz+t^3-t)(yz-ty-t^4+t^3)},
\end{align*}
so that $(x,y,z)=(\hat x,\hat y,\hat z) = P_u$, which completes the proof.
\end{proof}

\section{Positive-rank elliptic curves} \label{mainproof}

In this section we prove Theorem~\ref{sumprod}, by showing that certain
elliptic curves have positve rank.
  Our proof relies on Mazur's
theorem on rational torsion of elliptic curves \cite{M}:

\begin{thm}[Mazur] \label{Mazur}
For any elliptic curve $E$ over\/ $\Q$, the torsion subgroup of $E(\Q)$ is
isomorphic to either\/ $\Z/n\Z$ (with $1\le n\le 12$ and $n\ne 11$) or\/
$\Z/2\Z\oplus \Z/2n\Z$ (with $1\le n\le 4$).
\end{thm}

Recall that, for any $a,b,c\in\Q$, the set $S_{abc}(\Q)$ consists of all
triples $(x,y,z)$ of rational numbers such that
$x+y+z=a+b+c$ and $xyz=abc$.  Also, $E_{abc}$ is the curve in $\PP^2$ defined by the affine equation (\ref{Eabc}).
Finally, we write $\Sigma_{abc}$ for the set of permutations of the sequence $(a,b,c)$.
We will prove the following refinement of Theorem~\ref{sumprod}:

\begin{thm} \label{sharp}
Let $a,b,c$ be pairwise distinct nonzero rational numbers.
If $(a+b+c)^3=27abc$ then $E_{abc}$ has genus zero and $S_{abc}(\Q)$ is infinite.
If $(a+b+c)^3\ne 27abc$ then $E_{abc}$ is an elliptic curve which contains
the points in the set
\[
T_{abc}:=\left\{\Bigl(-AC+\frac{(A+B+C)^2}{12}, \,\frac{AC(C-A)}2\Bigr)\colon
(A,B,C) \in\Sigma_{abc}\right\},
\]
and the subgroup of $E_{abc}(\Q)$ generated by $T_{abc}$ is
\[
\begin{cases}
\Z/12\Z & \text{ if $A(B-C)^3=B(C-A)^3$ for some $(A,B,C)\in\Sigma_{abc}$}, \\
\Z/9\Z & \text{ if $AB^2+BC^2+CA^2=3ABC$ for some $(A,B,C)\in\Sigma_{abc}$}, \\
\Z\oplus \Z/3\Z & \text{ otherwise}.
\end{cases}
\]
\end{thm}

In light of Lemma~\ref{E}, Theorem~\ref{sumprod} follows at once from Theorem~\ref{sharp} and
the fact that $S_{abc}(\Q)$ is infinite when $abc=0$.  We now prove Theorem~\ref{sharp}.

\begin{proof}[Proof of Theorem~\ref{sharp}]
Let $a,b,c$ be pairwise distinct nonzero rational numbers.
By Lemmas~\ref{sing} and \ref{singpar}, the set $S_{abc}(\Q)$ is infinite
if $(a+b+c)^3=27abc$, and Lemma~\ref{genus} implies that $E_{abc}$ has
genus zero in this case.
Henceforth assume that $(a+b+c)^3\ne 27abc$, so
that (by Lemma~\ref{genus}) the curve $E_{abc}$ is
an elliptic curve.  Lemma~\ref{E} implies that $E_{abc}(\Q)$ contains $T_{abc}$.
For any permutation $(A,B,C)$ of $(a,b,c)$, write
\[
P_{ABC}:=\Bigl(-AC+\frac{(A+B+C)^2}{12}, \,\frac{AC(C-A)}2\Bigr).
\]
Then, in the group $E_{abc}(\Q)$, we have the relations $P_{CBA}=-P_{ABC}$
and $P_{CAB}=P_{ABC}+Q$, where $Q:=((a+b+c)^2/12,\,abc/2)$.  Crucially, we observe
that $Q$ has order $3$.  Writing $\Gamma_{abc}$ for the group generated by $T_{abc}$,
it follows that $\Gamma_{abc}=\langle P_{ABC},Q\rangle$ for any $(A,B,C)\in\Sigma_{abc}$.
In particular, if $\Gamma_{abc}$ is infinite then
$\Gamma_{abc}\cong \Z\oplus \Z/3\Z$.  Note that $P_{ABC}\ne P_{DEF}$ for any distinct $(A,B,C),(D,E,F)\in\Sigma_{abc}$,
since if $P_{ABC}$ and $P_{DEF}$ have the same $x$-coordinate then $AC=DF$ so $B=E$, whence
$P_{DEF}=P_{CBA}=-P_{ABC}\ne P_{ABC}$.
Next, considering $x$-coordinates shows that the group $\langle Q\rangle$ is
disjoint from $T_{abc}$, so $\#\Gamma_{abc}\ge 9$.  By Mazur's theorem, if $\Gamma_{abc}$
is finite then it must be either $\Z/9\/Z$, \,$\Z/12\Z$, or $\Z/2\Z\oplus\Z/6\Z$; in any
case, $\Gamma_{abc}$ has a unique subgroup of order $3$.
For any $(A,B,C)\in\Sigma_{abc}$, we compute that $2P_{ABC}$ equals
\[
\Bigl(\frac{(A+B+C)^2}{12} - \frac{AC(A-B)(B-C)}{(A-C)^2},\,
\frac{AC}{2(A-C)^3}(A(C-B)^3-C(B-A)^3)\Bigr).
\]
Examining $x$-coordinates shows that $2P_{ABC}\notin\langle Q\rangle$,
so the order of $P_{ABC}$ does not divide $6$.  It follows that
$\Gamma_{abc}\not\cong \Z/2\Z\oplus\Z/6\Z$.

We now determine all $a,b,c$ for which $\Gamma_{abc}\cong \Z/12\Z$.
First note that this occurs if and only if some $P_{ABC}$ has order $4$:
for, if $P_{ABC}$ has order $4$ then
$\Gamma_{abc}=\langle P_{ABC},Q\rangle\cong\Z/12\Z$, and if $\Gamma_{abc}\cong \Z/12\Z$
then some element of $T_{abc}$ has order $4$ because $\Z/12\Z$ contains only four elements
whose order is neither $4$ nor a divisor of $6$.
Next, $P_{ABC}$ has order $4$ if and only if $2P_{ABC}$ has order $2$; equivalently,
the $y$-coordinate of $2P_{ABC}$ is zero, which means that $A(C-B)^3=C(B-A)^3$.

Finally, we determine all $a,b,c$ for which $\Gamma_{abc}\cong \Z/9\Z$.
This occurs if and only if every $P_{ABC}$ has order $9$, which means that
$3P_{ABC}=\pm Q$.  Since $P_{CBA}=-P_{ABC}$, this says that some $P_{ABC}$ satisfies
$3P_{ABC}=Q$, or equivalently $2P_{ABC}=-P_{ABC}+Q$.  We compute
\[
-P_{ABC}+Q = \Bigl(\frac{(A+B+C)^2}{12}-AB, \,\frac 12 AB(B-A)\Bigr).
\]
Note that $2P_{ABC}\ne -(-P_{ABC}+ Q)$, since $P_{ABC}\ne Q$.
Thus, $-P_{ABC}+Q$ and $2P_{ABC}$ are equal if and only if they have the same
$x$-coordinate, which says that
\[
B(A-C)^2 = C(A-B)(B-C),
\]
or equivalently
\[
A^2B+B^2C+C^2A = 3ABC. \qedhere
\]
\end{proof}

\section{The remaining cases of Question~\ref{q}} \label{fun}

In this section we discuss Question~\ref{q} in the cases where either (\ref{first})
or (\ref{second}) does not hold.  We first show that (\ref{first}) and (\ref{second}) automatically hold in
certain situations, and use this to prove Corollary~\ref{coprime}.

\begin{lemma} \label{lc}
If $a,b,c$ are nonzero integers which are pairwise coprime and pairwise distinct,
then $a(b-c)^3\ne b(c-a)^3$ and $ab^2+bc^2+ca^2\ne 3abc$.
\end{lemma}

\begin{proof}
First assume that $a(b-c)^3=b(c-a)^3$, so $a\mid b(c-a)^3$.
Since $a$ is coprime to $b$ and $c$, it is coprime to $b(c-a)^3$, so
$a\in\{-1,1\}$.  Likewise, $b\in\{1,-1\}$, so $b=-a$.  Then $a(b-c)^3=b(c-a)^3=a(a-c)^3$,
so that $b-c=a-c$ and thus $b=a$, a contradiction.

Next assume that $ab^2+bc^2+ca^2=3abc$.  Then $a\mid bc^2$, and since $a$ is
coprime to $b$ and $c$, it follows that $a\in\{1,-1\}$.  Likewise, both
$b$ and $c$ must be in $\{1,-1\}$, so $a,b,c$ cannot be pairwise distinct.
\end{proof}

Corollary~\ref{coprime} follows at once from this result and Theorem~\ref{sumprod}, together with
the fact that (\ref{first}) and (\ref{second}) hold for every permutation $(A,B,C)$ of $(-1,0,1)$.

Next we determine all $(a,b,c)$ for which either (\ref{first}) or (\ref{second}) does not hold.

\begin{prop} \label{anotherparam}
The triples $(a,b,c)$ of pairwise distinct nonzero rational numbers
such that $a(b-c)^3=b(c-a)^3$ are
\[
(r(t+1)^3, \,-rt^3, \,-rt(t+1)(2t^2+2t+1))
\]
where $r\in\Q^*$ and $t\in\Q\setminus\{-1,-\frac 12,0\}$.  The triples $(a,b,c)$ of
pairwise distinct nonzero rational numbers such that $ab^2+bc^2+ca^2=3abc$ are
\[
(rt^2, \,-r(t+1), \,rt(t+1)^2)
\]
where $r\in\Q^*$ and $t\in\Q\setminus\{-1,0\}$.  In both cases,
the pair $(r,t)$ is uniquely determined by the triple $(a,b,c)$.
\end{prop}

\begin{proof}
Let $a,b,c$ be pairwise distinct nonzero rational numbers such that
$a(b-c)^3=b(c-a)^3$.  Then $t:=(b-c)/(a-b)$ is a nonzero rational number.
Further, $t\ne -1$ since otherwise $b-c=b-a$ implies $a=c$.
Finally, $t\ne -\frac 12$, since otherwise $2(b-c)=b-a$ implies $b-c=c-a$, so the identity
$a(b-c)^3=b(c-a)^3$ reduces to $a=b$.  Thus $t\in\Q\setminus\{-1,-\frac 12,0\}$, and
for $r:=a/(t+1)^3$ we compute
\[
b+rt^3=\frac{a(b-c)^3-b(c-a)^3}{(a-c)^3}
\]
and
\[
c+rt(t+1)(2t^2+2t+1)=\frac{a(b-c)^3-b(c-a)^3}{(a-b)(a-c)^2},
\]
so $(a,b,c)=(r(t+1)^3, \,-rt^3, \,-rt(t+1)(2t^2+2t+1))$.
Conversely, this last equation implies that $ab^{-1}=-(1+t^{-1})^3$, so that $t$ (and
hence $r$) is uniquely determined by $a$ and $b$; moreover, for any $r\in\Q^*$ and $t\in\Q\setminus\{-1,-\frac 12,0\}$,
if we define $a,b,c$ by this last equation then $a,b,c\in\Q^*$ are pairwise distinct and $a(b-c)^3=b(c-a)^3$.

Now let $a,b,c$ be pairwise distinct nonzero rational numbers such that
$ab^2+bc^2+ca^2=3abc$.  Then $t:=(a-c)/(b-a)$ is a nonzero rational number,
and $t\ne -1$ since $b\ne c$.  For $r:=a/t^2$ we compute
\[
b+r(t+1)=\frac{ab^2+bc^2+ca^2-3abc}{(a-c)^2}
\]
and
\[
c-rt(t+1)^2=\frac{ab^2+bc^2+ca^2-3abc}{(a-b)(a-c)},
\]
so $(a,b,c)=(rt^2, \,-r(t+1), \,rt(t+1)^2)$.
Conversely, this last equation implies that $acb^{-2}=t^3$, so that $t$
(and hence $r$) is uniquely determined by $(a,b,c)$; moreover, for any $r\in\Q^*$ and any $t\in\Q\setminus\{-1,0\}$,
if we define $a,b,c$ by this last equation then $a,b,c\in\Q^*$ are pairwise distinct and $ab^2+bc^2+ca^2=3abc$.
\end{proof}

Next we show that the failure of (\ref{first}) or (\ref{second}) does not
determine whether $S_{abc}(\Q)$ is infinite.

\begin{example}
One can check that $E_{abc}(\Q)$ is finite when $(a,b,c)$ is either $(3,10,24)$ or $(1,-2,4)$,
and infinite when $(a,b,c)$ is either $(2,15,54)$ or $(-3,4,18)$.
Here $(a,b,c)=(3,24,10)$ and $(2,54,15)$ violate (\ref{first}), but every
permutation $(A,B,C)$ of either of these triples satisfies (\ref{second}).
On the other hand, $(a,b,c)=(1,-2,4)$ and $(-3,18,4)$ violate (\ref{second}),
but every permutation $(A,B,C)$ of either of these triples satisfies (\ref{first}).
\end{example}

In the spirit of existing conjectures (e.g.\ from \cite{BMSW}), and in the absence of any reason to believe otherwise,
it seems reasonable to guess that $S_{abc}(\Q)$ is infinite for half of all triples $(a,b,c)$ of
nonzero rational numbers such that either $a(b-c)^3=b(c-a)^3$ or $ab^2+bc^2+ca^2=3abc$,
when triples are ordered by the largest absolute value of any integer occurring as either a numerator or denominator
of any rational number in the triple.  We used Magma's
non-rigorous calculation of analytic ranks of elliptic curves to compute the analytic
rank of $E_{abc}$ for all triples $(a,b,c)$ of nonzero pairwise coprime rational numbers
which violate either (\ref{first}) or (\ref{second}) and whose numerator and denominator
have absolute value at most $30$.  There are $1801$ such triples, and Magma suggests that
the analytic rank of $E_{abc}$ is zero for $783$ (or about $43.48\%$) of them.  By Lemma~\ref{E}
and the Birch--Swinnerton-Dyer conjecture, the analytic rank of $E_{abc}$ is zero precisely
when $S_{abc}(\Q)$ is finite.  If we so desire, we can avoid assuming the Birch--Swinnerton-Dyer conjecture here
by restricting to cases where the analytic rank of $E_{abc}$ is at most one, since the Birch--Swinnerton-Dyer conjecture
is known to be true in those cases by results of Gross--Zagier \cite{GZ} and Kolyvagin \cite{K}, together with \cite{BCDT}
and either \cite{BFH} or \cite{MM}.  
Although $43.48\%$ is somewhat less than $50\%$, it is
closer to $50\%$ than is usual for data involving ranks of elliptic curves, so at least we can say
that our guess is more consistent with the data than are well-established conjectures of the same
flavor \cite{BMSW}.

\section{Positive solutions} \label{secpos}

In this section we examine positive rational solutions of the system $x+y+z=a+b+c$, \,$xyz=abc$.
We will use the Poincar\'e--Hurwitz theorem (\cite[Satz 13]{H}; see also \cite[p.~173]{P}):

\begin{lemma}[Poincar\'e--Hurwitz] \label{PH}
Let $E$ be a nonsingular cubic curve in $\PP^2$ which is defined over/\ $\Q$.  If the set $E(\Q)$ is infinite,
then every open subset of $\PP^2(\R)$ which contains one point of $E(\Q)$ must contain infinitely many points of $E(\Q)$.
\end{lemma}

We now prove a refined version of Corollary~\ref{sumprodpos}, which will be needed in the next section.

\begin{lemma} \label{AMGM}
Let $a,b,c$ be pairwise distinct positive rational numbers such that every permutation $(A,B,C)$ of $(a,b,c)$
satisfies $A(B-C)^3\ne B(C-A)^3$.  Then $S_{abc}(\Q)$ contains infinitely many points in any open subset of\/ $\R^3$ which
contains $(a,b,c)$.
\end{lemma}

\begin{proof}
Since $a,b,c$ are distinct and positive, their arithmetic mean is greater than their
geometric mean, so $(a+b+c)^3>27abc$.  Likewise, for any permutation $(A,B,C)$ of $(a,b,c)$,
comparing the arithmetic and geometric means of $AB^2$, $BC^2$ and
$CA^2$ shows that $AB^2+BC^2+CA^2\ge 3ABC$, with equality occurring if and only if $AB^2=BC^2=CA^2$.
This equality condition implies that $A^2B^4=(AB^2)^2=(BC^2)(CA^2)=A^2BC^3$, so that $B^3=C^3$, which is impossible
since $B,C$ are distinct rational numbers.  Thus, Theorem~\ref{sharp} tells us that $E_{abc}$ is an elliptic curve containing infinitely
many rational points, so by Lemma~\ref{PH} the set $E_{abc}(\Q)$ has infinite intersection with any neighborhood in
$\PP^2(\R)$ of any point $P\in E_{abc}(\Q)$.  Since the map $\rho$ from Lemma~\ref{E} is a homeomorphism from $S_{abc}(\R)$
to $E_{abc}(\R)\setminus I_{abc}$, it follows that $S_{abc}(\Q)$ has infinite intersection with any neighborhood in $\R^3$ of 
$\rho^{-1}(P)$ if $P\notin I_{abc}$.  Taking $P=\rho((a,b,c))$ yields the result.
\end{proof}

Corollary~\ref{sumprodpos} follows from Lemma~\ref{AMGM} by taking the open set to be an open ball centered at $(a,b,c)$ of
radius less than the smallest of $a,b,c$.  Next, Corollary~\ref{coprimepos} follows at once from Corollary~\ref{sumprodpos} and
Lemma~\ref{lc}.


\section{Equal sums and equal sums of cubes} \label{seccube}

In this section we analyze the system of equations $x+y+z=a+b+c$, \,$x^3+y^3+z^3=a^3+b^3+c^3$
for fixed $a,b,c\in\Q$.  This system has been studied at least since 1915 \cite{Ger}, and more recently in the papers
\cite{B,BB,BL,C,C2,L,Ren}, inspired in part by the occurrence of this system in the physics literature
in the context of zeros of the $6j$ Racah coefficients \cite{BL}.

We use a substitution from \cite{BL} (in slightly modified form) to transform
this system into the system $u+v+w=d+e+f$, \,$uvw=def$ for certain $d,e,f\in\Q$.
For any field $K$ with $\charp(K)\ne 2$, define
$\psi\colon K^3\to K^3$ and $\phi\colon K^3\to K^3$ via
\begin{align*}
\psi((x,y,z))&=\Bigl(\frac{y+z}2, \,\frac{x+z}2, \,\frac{x+y}2\Bigr) \\
\phi((x,y,z))&=(-x+y+z, \,x-y+z, \,x+y-z).
\end{align*}
For fixed $a,b,c\in\Q$, let $U_{abc}$ be the variety defined by
$x+y+z=a+b+c$ and $x^3+y^3+z^3=a^3+b^3+c^3$.

\begin{lemma} \label{coord}
The functions $\psi$ and $\phi$ are bijective and inverse to one another.  For any $a,b,c\in\Q$, we have
$U_{abc}(K)=\phi(S_{\psi((a,b,c))}(K))$ and $S_{abc}(K)=\psi(U_{\phi((a,b,c))}(K))$.
\end{lemma}

\begin{proof}
It is easy to check that both $\phi\circ\psi$ and $\psi\circ\phi$ are the identity map on $K^3$,
which implies
that they are inverses and they are both bijective.  Letting $s\colon K^3\to K$ be the map $s((x,y,z))=x+y+z$, we see that
$s\circ\phi=s=s\circ\psi$.
Pick any $a,b,c,x,y,z\in K$ such that $x+y+z=a+b+c$.
Let $(u,v,w)=\psi((x,y,z))$ and $(d,e,f)=\psi((a,b,c))$, so that also $(x,y,z)=\phi((u,v,w))$ and $(a,b,c)=\phi((d,e,f))$.
Then we have 
\[
x^3+y^3+z^3=(-u+v+w)^3+(u-v+w)^3+(u+v-w)^3=(u+v+w)^3-24uvw,
\]
and likewise
$a^3+b^3+c^3=(d+e+f)^3-24def$.  Since 
\[
u+v+w=x+y+z=a+b+c=d+e+f,\] it follows that
$x^3+y^3+z^3$ and $a^3+b^3+c^3$ are equal if and only if $uvw$ and $def$ are equal.
Thus $U_{abc}(K)=\phi(S_{def}(K))$ and $S_{def}(K)=\psi(U_{abc}(K))$.
\end{proof}

We conclude this paper with proofs of Propositions~\ref{cube} and \ref{poscube}.

\begin{proof}[Proof of Proposition~\ref{cube}]
Write $(d,e,f):=\psi((a,b,c))$, so that Lemma~\ref{coord} exhibits a bijection
between $U_{abc}(\Q)$ and $S_{def}(\Q)$.  Since $a,b,c$ are pairwise distinct, also
$d,e,f$ are pairwise distinct.  By Theorem~\ref{sumprod}, the set $S_{def}(\Q)$ is
infinite so long as every permutation $(D,E,F)$ of $(d,e,f)$ satisfies both
$D(E-F)^3\ne E(F-D)^3$ and $DE^2+EF^2+FD^2\ne 3DEF$.  These hypotheses are
equivalent to the assertion that (\ref{third}) and (\ref{fourth}) hold for
every permutation $(A,B,C)$ of $(a,b,c)$, so the result follows.
\end{proof}

\begin{proof}[Proof of Proposition~\ref{poscube}]
Write $(d,e,f):=\psi((a,b,c))$, so $d,e,f$ are pairwise distinct positive rational
numbers.  Our hypothesis on $a,b,c$ implies that $D(E-F)^3\ne E(F-D)^3$ for every permutation $(D,E,F)$ of $(d,e,f)$.
Thus, by Lemma~\ref{AMGM}, the set $S_{def}(\Q)$ contains infinitely many points in any open subset of $\R^3$ which
contains $(d,e,f)$.  Since $\phi$ is a homeomorphism from $\R^3$ to itself, and $(a,b,c)=\phi((d,e,f))$ is in $\phi(S_{def}(\Q))$,
it follows that $\phi(S_{def}(\Q))$ contains infinitely many points in any open subset of $\R^3$ which contains $(a,b,c)$.
Finally, Lemma~\ref{coord} shows that $\phi(S_{def}(\Q))=U_{abc}(\Q)$, so the result follows by choosing the open set to be
an open ball centered at $(a,b,c)$ of radius less than the smallest of $a,b,c$.
\end{proof}

\end{document}